\documentclass[preprint,12pt
]{elsarticle}

\usepackage{amsthm}
\usepackage{graphicx,amscd,amsmath,amsfonts,amssymb} %,geometry}
\journal{arXiv.org}

\begin{document}

\newtheorem{theorem}{Theorem}[section]
\newtheorem{lemma}[theorem]{Lemma}
\newtheorem{proposition}[theorem]{Proposition}
\newtheorem{corollary}[theorem]{Corollary}
\newtheorem{definition}{Definition}[section]

{\theoremstyle{definition}
\newtheorem{remark}[theorem]{Remark}
\newtheorem{example}[theorem]{Example}
\newtheorem{remarks}[theorem]{Remarks}
}

\hyphenation{con-ti-nuous}
\hyphenation{Theo-rem}

\def\N{\mathbb N}
\def\C{\mathbb C}
\def\R{\mathbb R}
\def\D{\mathbb D}
\def\T{\mathbb T}
\def\Q{\mathbb Q}
\def\la{\lambda}
\def\dis{\displaystyle}
\def\ovl{\overline}
\def\wt{\widetilde}
\def\eps{\varepsilon}

\begin{frontmatter}



\title{Algebraic structure of continuous, unbounded and integrable functions}


\author{M.C.~Calder\'on-Moreno\fnref{sup}}
\ead{mccm@us.es}
\author{P.J.~Gerlach-Mena\fnref{sup}}
\ead{gerlach@us.es}
\author{J.A.~Prado-Bassas\corref{cor1}\fnref{sup}}
\ead{bassas@us.es}
\cortext[cor1]{Corresponding author}
\fntext[sup]{The authors have been partially supported by the Plan Andaluz de Investigaci\'on de la Junta de Andaluc\'{\i}a FQM-127, Grant P08-FQM-03543 and by MINECO Grant MTM2015-65242-C2-1-P.}

\address{Dpto. An\'alisis Matem\'atico, Fac. Matem\'aticas, Universidad de Sevilla. \\
Avda.~Reina Mercedes s/n, 41080 Sevilla, Spain.}

\begin{abstract}
In this paper we study the large linear and algebraic size of the family of unbounded continuous and integrable functions in $[0,+\infty)$ and of the family of sequences of these functions converging to zero uniformly on compacta and in $L^1$-norm. In addition, we concentrate on the speed at which these functions grow, their smoothness and the strength of their convergence to zero.
\end{abstract}


\end{frontmatter}

% \linenumbers


\section{Introduction}

Contrary to what happens with a series of real numbers, there are even integrable functions on $[0,+\infty)$ that do not converge pointwise to zero as $x\to+\infty$. In fact, it is easy to construct unbounded, continuous and integrable functions on $[0,+\infty)$ (see Example \ref{ejemplo1}). In this paper we will analyze the existence of large algebraic structures of sets of such functions and of sequences of these functions converging to zero, in a manner to be specified later.

We will first establish some notation. As is standard, we will denote by $\N$ and $\R$ the set of positive integers and the real line, respectively. For any interval $I\subset[0,+\infty)$, by ${\mathcal C}(I)$ we will denote the vector space of all continuous functions on $I$; in the special case $I=[0,+\infty)$, the space ${\mathcal C}([0,+\infty))$ will be endowed with the topology of the uniform convergence on compacta, under which ${\mathcal C}([0,+\infty))$ becomes a complete metrizable topological vector space or $F$-space. Also, if $I\subset [0,+\infty)$ is any interval, $L^1(I)$ will stand for the space of all Lebesgue integrable functions on $I$, endowed with its usual topology of the $L^1$-norm $\|f\|_I=\int_I|f(x)|dx$.

The general aim of this paper is the search for linear structures within a nonlinear setting, a trending branch that has attracted the attention of many mathematicians within the last decade. The concept of {\em lineability} was coined by the Russian mathematician V.I. Gurariy \cite{gurariy1996}. He later partnered with R.M. Aron, J.B. Seoane-Sep\'ulveda {\em et al.} \cite{arongarciaperezseoane2009}, the concept of \emph{algebrability} (see \cite{aronperezseoane2006}), which relates to the algebraic size of a set. We refer to the survey \cite{bernalpellegrinoseoane2014} and the book \cite{aronbernalpellegrinoseoane2015} for a wide background on this topic. Specifically, we introduce the next definitions that will be used later.

Given a vector space $X$ and a subset $A \subset X$, we say that $A$ is:
\begin{itemize}
\item lineable whenever there is an infinite dimensional vector space $M$ such that $M \setminus \{0\} \subset A$;
\item $\alpha$-lineable if, in addition, $\text{dim}(M) = \alpha$;
\item maximal lineable if $A$ is $\text{dim}(X)$-lineable.
\end{itemize}
Moreover, if $X$ is a topological vector space, then $A$ is said to be:
\begin{itemize}
\item dense-lineable if there is a dense vector space $M$ with $M \setminus \{0\} \subset A$;
\item maximal dense-lineable if we also have that $\text{dim}(M) = \text{dim}(X)$.
\end{itemize}
When $X$ is a topological vector space contained in some (linear) algebra, then $A$ is called:
\begin{itemize}
\item algebrable if there is an algebra $M$ so that $M \setminus \{0\} \subset A$ and $M$ is infinitely generated, that is, the cardinality of any system of generators of $M$ is infinite;
\item strongly algebrable if the algebra $M$ above can be taken free.
\end{itemize}
If the algebraic structure of $X$ is commutative, the strong algebrability is equivalent to the existence of a generating set $B$ of the algebra $M$ such that for any $s \in \N$, any nonconstant polynomial $P$ in $s$ variables and any distinct $f_1,\ldots,f_s \in B$, we have $P(f_1, \ldots ,f_s) \in M\setminus\{0\}$.

Focusing our attention on lineability of continuous functions, Gurariy \cite{gurariy1996} stated the lineability of the set of functions that are continuous on $[0,1]$ but nowhere differentiable. Later in 2013 (see \cite{jimenezmunozseoane2013}), Jim\'enez-Rodr\'{\i}guez, Mu\~noz-Fern\'andez and Seoane-Sep\'ulveda gave the first constructive proof of the $\mathfrak{c}$-lineability of the so called Weierstrass' Monsters, while Bayart and Quarta \cite{bayartquarta2007} proved the dense algebrability (that is, the existence of a dense infinitely generated algebra) of the more restrictive family of continuous and nowhere H\"older functions on $[0,1]$.

In the same vein, the first and third author, jointly with Bernal \cite{bernalcalderonprado2015}, proved the maximal lineability and strong algebrability of the set of continuous functions $\varphi:[0,1]\to\R^2$ whose image has nonempty interior and a zero-(Lebesgue) measure boundary as well as the maximal dense-lineability of the family of continuous functions from $[0,1]$ to $\R^2$ whose images have positive Lebesgue measure. As a related result, in 2014 Albuquerque \cite{albuquerque2014} states the maximal lineability of the set of continuous surjections from $\R$ to $\R^2$. Extensions of these results can be found in \cite{albuquerquebernalpellegrinoseoane2014,bernalordonez2014}.

But a large linear structure cannot always be obtained for continuous functions with additional properties. In 2004, Gurariy and Quarta showed that the set of continuous functions on $[0,1]$ which admit one and only one absolute maximum is 1-lineable but not 2-lineable. In addition, they also proved that the set of continuous functions on the whole real line satisfying the mentioned property is 2-lineable, while the corresponding subset of continuous functions vanishing at infinity is not 3-lineable.

Turning now to the framework of integrable functions, Mu\~noz, Palmberg, Puglisi and Seoane \cite{munozpalmbergpuglisiseoane2008} proved the $\mathfrak{c}$-lineability of $L^p[0, 1]\setminus L^q[0, 1]$ for $1\le p<q$ and of $L^p(I)\setminus L^q(I)$ for $p>q\ge 1$ and any unbounded interval $I\subset\R$. In \cite{arongarciaperezseoane2009,bernal2010,bernalordonez2013,bernalordonez2014, botelhocariellofavaropellegrinoseoane2013,botelhofavaropellegrinoseoane2012,kitsontimoney2011} some generalizations and extensions of these results are proved.

Concerning the relationship between lineability and integrability, Garc\'{\i}a, Mart\'{\i}n and Seoane showed in 2009 \cite{garciamartinseoane2009} that, given an arbitrary unbounded interval $I\subset \R$, the set of all almost everywhere continuous bounded functions on $I$ which are also not Riemann integrable, contains an infinitely generated closed subalgebra. Moreover, they also proved the existence of an infinite dimensional closed vector space consisting (except for zero) of continuous and bounded functions on an unbounded interval $I$ which are not Riemann integrable; in particular, this set is $\mathfrak{c}$-lineable. In addition, they also state the lineability of the set of Riemann integrable functions (on an unbounded interval) which are not Lebesgue integrable and the $\mathfrak{c}$-lineability (in fact, the existence of a closed infinite dimensional subspace) of the set of Lebesgue integrable functions that are not Riemann integrable. In this direction, Bernal and Ord\'o\~nez \cite{bernalordonez2014} show that the set of continuous and Riemann integrable functions on $[0,+\infty)$ which do not belong to $L^p([0,+\infty))$ for any $0<p<+\infty$ is maximal dense-lineable.

On the other hand, Garc\'{\i}a, Grecu, Maestre and Seoane \cite{garciagrecumaestreseoane2010} were able to show the existence of an infinite dimensional Banach space of bounded and Lebesgue integrable functions in $[a,b]$ with antiderivatives at every point of the interval but, except for zero, are not Riemann integrable.

Finally, in the setting of sequences of functions, there are, up to the authors knowledge, only few and relatively recent results. In fact, in 2014, Bernal and Ord\'o\~nez proved in \cite{bernalordonez2014} that the family of sequences $(f_n)_n\subset\left(\R^\R\right)^\N$ of continuous bounded and integrable functions such that $\|f_n\|_\infty\overset{n\to\infty}{\longrightarrow}{+\infty}$, $\sup\{\|f_n\|_{L^1}:\,n\in\N\}<+\infty$ but $\|f_n\|_{L^1}\not\to 0$ ($n\to\infty$) is maximal lineable. Later, in 2017, Ara\'ujo, Bernal, Mu\~noz, the third author and Seoane \cite{araujobernalmunozpradoseoane2017} were able to show the $\mathfrak{c}$-lineability of the family of sequences of Lebesgue measurable functions $\R\to\R$ such that $f_n$ converges pointwise to zero and $f_n(I)=\R$ for any non-degenerate interval $I\subset\R$ and any $n\in\N$. Additionaly, the maximal dense-lineability (in the vector space of sequences of Lebesgue measurable functions $[0,1]\to\R$) of the family of sequences of Lebesgue measurable functions such that $f_n\to0$ in measure but not almost everywhere in $[0,1]$ is proven.

The aim of this paper is to contribute to the study of lineability and algebrability of continuous and integrable functions and of sequences of them. For that purpose, in Section 2 we focus on the linear and algebraic size of the set of continuous, unbounded and integrable functions on $[0,+\infty)$, paying special attention to the natural topology carried by the framework space. In Section 3, we turn to the problem of the algebraic genericity of the set of sequences of functions in the just mentioned class that tend to zero both in $L^1$-norm and uniformly on compacta. We conclude the paper with a final section devoted to extend some results in the preceding sections by looking for a stronger way of convergence of sequences, the possible rate of growth of the functions, or the smoothness of them.

\section{Unbounded, continuous and integrable functions}

As we said in the Introduction, in the context of continuous and integrable functions on $[0,+\infty)$, intuition leads us to believe that these functions should go somehow to zero when $x$ tends to infinity. This is clearly true if $f$ is decreasing or if it exists $\displaystyle{\lim_{x\to +\infty} f(x)}$, but in general it is not difficult to provide examples of continuous and integrable functions on $[0,+\infty)$ that not only are not close to zero but also satisfy $\displaystyle{\limsup_{x\to +\infty} |f(x)}| = +\infty$. The next example shows a classical undergraduate construction of an unbounded, continuous and integrable function on $[0,+\infty)$ that will be very useful going forward.

\begin{example}\label{ejemplo1}
For any $n \in \N$, we consider the triangular function $T_n : [0,+\infty) \longrightarrow \R$ given by:

$$T_n(x) =
\begin{cases}
n(2^{n+1}x + (1-n2^{n+1}))  & \text{if } x \in \left[n-\tfrac{1}{2^{n+1}},n\right), \\
n(-2^{n+1}x + (1+n2^{n+1})) & \text{if } x \in \left[n,n+\tfrac{1}{2^{n+1}}\right], \\
0                           & \text{otherwise,}
\end{cases}
$$
and the function $f: [0,+\infty) \longrightarrow \R$ defined by joining the previous triangles:
\begin{equation*}\label{EQ1}
f(x) = \sum_{n=1}^{\infty} T_n(x).
\end{equation*}
That is, $f$ ``draws'' all the triangles of height $n$, supported on intervals centered at $n$ with length $\frac{1}{2^n}$ ($n\in\N$).

The function $f$ is continuous by construction, and also integrable. Indeed,
$$\|f\|_{L^1}=\int_0^{\infty}|f(x)|dx=\sum_{n=1}^{\infty} \int_0^{\infty} T_n(x)dx = \sum_{n=1}^{\infty} \dfrac{n}{2^{n+1}}  < +\infty.$$
Furthermore, $f$ is unbounded just by considering the sequence $x_n= n$:
$$f(x_n) = f(n) = n \overset{n\to\infty}{\longrightarrow} +\infty.$$

Observe that with a similar construction we could get the value $\|f\|_{L^1} \in (0,+\infty)$ arbitrarily small, just adjusting the size of the bases of the triangles.
\end{example}

\bigskip

From now on we denote by ${\mathcal A}$ the family of unbounded continuous integrable functions in $[0,+\infty)$, that is,
$$\mathcal{A} = \left\{ f \in \mathcal{C}([0,+\infty)) \cap L^1([0,+\infty)):\,\, \limsup_{x \to  +\infty} |f(x)| = +\infty \right\}.$$

The natural question that now arises is how many functions are there in $\mathcal{A}$ and what the meaning of ``many'' would be.

Observe that ${\mathcal A}$ is not a vector space, just take $f(x)$ and $e^{-x}-f(x)$, where $f\in{\mathcal A}$. Observe also that the maximal dimension of a vector space $M\subset {\mathcal A}\cup\{0\}$ is $\mathfrak{c}$, where $\mathfrak{c}=$the dimension of continuum.

In order to study the algebraic size of the set of these functions we need the next result that will allow us to get dense-lineability. Recall that given two sets $A,\,B$ of a vector space, then $A$ is said to be stronger than $B$ if $A+B\subset A$ (see \cite{arongarciaperezseoane2009,bernalordonez2014}).

\begin{lemma}\label{denso}
Let $X$ be a metrizable topological vector space, $A \subset X$ maximal lineable and $B \subset X$ dense-lineable in $X$ with $A \cap B = \varnothing$. If $A$ is stronger than $B$ then $A$ is maximal dense-lineable.
\end{lemma}

\begin{theorem}\label{T1}
The family $\mathcal{A}$ is maximal lineable.
\end{theorem}

\begin{proof}

We are looking for a vector subspace $M$ of dimension $\mathfrak{c}$ such that $M \subset \mathcal{A} \cup \{0\}$. Consider the triangular function $f$ given in the Example \ref{ejemplo1} and for any $t \in \left[0,\frac{1}{8}\right)$ we define the functions:
$$f_t(x) := f(x-t).$$

We have $f_t \in \mathcal{A}$ for each $t \in \left[0,\frac{1}{8}\right)$. Define $M$ as the set
$$M := \text{span} \left\{ f_t \ : \ t \in \left[0,\tfrac{1}{8}\right)\right\}.$$
We claim that $M$ is our searched vector subspace. Let us assume that there are $0 < t_1 < t_2 < \ldots < t_s < \frac{1}{8}$ and scalars $c_1, c_2, \ldots, c_s \in \R$ not all simultaneously zero such that
$$ c_1 f_{t_1}(x) + c_2 f_{t_2}(x) + \ldots + c_sf_{t_s}(x) = 0 \quad (x \in [0,+\infty)).$$
Without loss of generality we can assume $c_1 \neq 0$. But then, for any $x_0 \in \left(\frac{3}{4} +t_1, \frac{3}{4} +t_2 \right)$, we get
$$c_1 f_{t_1}(x_0) + c_2 f_{t_2}(x_0) + \ldots + c_sf_{t_s}(x_0)= c_1 f_{t_1}(x_0) \neq 0.$$
So, the set $\{f_t:$ $t\in[0,\frac{1}{8})\}$ is linearly independent and $\text{dim}(M) = \mathfrak{c}$.

It is obvious that every finite combination of elements of $M$ is continuous and integrable, hence it only remains to prove the unboundedness of each mentioned linear combination $c_1f_{t_1}+\ldots +c_sf_{t_s}$. Note that we can assume $c_s\ne0$. Since $t_s > t_i$ for all $i =1,2,\ldots,s-1$, there is $N \in \N$ such that the support of the triangles of $f_{t_s}$ is disjoint with the support of the rest of triangles of $f_{t_i}$. Indeed, there is $N\in\N$ such that $$\min\{t_s-t_i:\,i=1,\ldots,s-1\}>\frac{1}{2^N}.$$ So $N+\frac{1}{2^{N+1}} + t_i < N - \frac{1}{2^{N+1}} + t_s$, for any $i=1,\ldots,s-1$.

Thus, by taking for any $n>N$ the point $x_n: = n +t_s$, we have
$$|c_1f_{t_1}(x_n) + c_2 f_{t_2}(x_n) + \ldots + c_sf_{t_s}(x_n)| = |c_s| n \overset{n \to \infty}{\longrightarrow} +\infty.$$
Hence the family $\mathcal{A}$ is maximal-lineable.
\end{proof}

Recall that $X :=\mathcal{C}([0,+\infty)) \cap L^1([0,+\infty))$ is a topological vector space when we endowed it with the natural translation-invariant distance given by
\begin{equation}
\label{metrica}
d_X(f,g) := \|f-g\|_{L^1} + \sum_{n=1}^{\infty} \dfrac{1}{2^n}\cdot \dfrac{\|f-g\|_{\infty,n}}{1+\|f-g\|_{\infty,n}},\end{equation}
where $\|f\|_{\infty,n} = \max\{|f(x)|:\, x\in[0,n]\}$. Note that $d_X$-convergence means $L^1$-convergence plus uniform convergence on compacta.

\begin{theorem}\label{T2}
The family $\mathcal{A}$ is maximal dense-lineable.
\end{theorem}

\begin{proof}[Proof]
By Theorem \ref{T1} we already have the maximal lineability of the family $\mathcal{A}$. It only remains to find an appropriate subset $B \subset \mathcal{C}([0,+\infty)) \cap L^1([0,+\infty))$ such that $B$ is dense-lineable and $\mathcal{A}$ is stronger than $B$.

Consider the family $B$ of all functions of the form
$$b_{p,n,\gamma}(x) =\begin{cases}
p(x) & \text{if } 0 \leq x \leq n, \\
\frac{p(n)}{\gamma}(n+\gamma-x) & \text{if } n < x \leq n+\gamma, \\
0 & \text{if } x > n+\gamma,
\end{cases}
$$
where $p(x)$ is a polygonal, $n \in \N$ and $\gamma > 0$. Recall that a polygonal is a continuous function consisting of finitely many affine linear mappings on compact subintervals of $[0,+\infty)$.

It is obvious that $B \subset \mathcal{C}([0,+\infty)) \cap L^1([0,+\infty))$ and that is a vector space. We are going to see that $B$ is dense in $X$. Let $f \in \mathcal{C}([0,+\infty)) \cap L^1([0,+\infty))$ and $\eps > 0$. Then there is $N \in \N$ such that
$$\sum_{k=N+1}^{\infty} \dfrac{1}{2^k} < \dfrac{\eps}{6}$$
and
$$\int_N^{\infty} |f(x)|dx < \dfrac{\eps}{6}.$$

By using uniform continuity, it is easy to see that the set of all polygonals is dense in $\mathcal{C}([0,N])$ even with the property that the approximating polygonal matches with the approximation function at the end of the interval. Consequently, we can take a polygonal $p(x)$ in $[0,N]$ such that $p(N) = f(N)$ and $\|f-p\|_{\infty,N} < \frac{\eps}{6N}$. Now define $\gamma:=\frac{\eps}{6(1+|f(N)|)}$. Then
\begin{eqnarray*}
d_X(f,b_{p,N,\gamma}) & \leq &
\|f-p\|_{\infty,N} \sum_{k=1}^{N} \frac{1}{2^k} + \sum_{k=N+1}^{\infty} \frac{1}{2^k} + \int_0^N |f(x) - p(x)|dx \\
&&+ \int_{N}^{N+\gamma} \left| f(x) - \tfrac{f(N)}{\gamma}(N-x+\gamma)\right|dx + \int_{N+\gamma}^{\infty} |f(x)|dx \\
& < & \dfrac{\eps}{6}\cdot 1 + \dfrac{\eps}{6} + \dfrac{\eps}{6N}\cdot N + \int_N^{\infty} |f(x)|dx + \tfrac{|f(N)|}{\gamma}\cdot \gamma\cdot\gamma+\dfrac{\eps}{6} < \eps.
\end{eqnarray*}

Hence, the set $B$ a dense vector space and, in particular, is dense-lineable. By its construction it is clear that any function in $B$ is bounded, so $\mathcal{A} \cap B = \varnothing$ and $\mathcal{A} + B \subset \mathcal{A}$.
Finally, Lemma \ref{denso} gives us the maximal dense-lineability of $\mathcal{A}$.
\end{proof}

In the following, we are going to establish the algebrability of $\mathcal{A}$. But, before this, we need to introduce some notation.
Given a (monic) monomial $m(x_1,\ldots, x_s) = \prod_{i=1}^s x_i^{\alpha_i}$, where $s\in\N$ and $\alpha_i\in\N\cup\{0\}$, and the increasing sequence of prime numbers $p = (p_i)_i$, we define the $p$-index of $m$ by
$$\text{ind}_p(m) = m(p_1,\ldots, p_s) = \prod_{i=1}^s p_i^{\alpha_i}.$$

\begin{remark}\label{obs1}
Observe that the uniqueness of the factorization theorem states that, given a natural number $n \in \N$, there is only one way of expressing it as product of powers of prime numbers. That is, there is an unique factorization
$$n = \prod_{i=1}^s p_i^{\alpha_i} = p\text{-index of a monomial }m,$$
which allows us to state a bijection between $\N$ and the set of all monic monomials. That means that, given a monomial $m$, it is uniquely described by its $p$-index $\text{ind}_p(m)$.
\end{remark}

\begin{theorem}\label{T3}
The family $\mathcal{A}$ is strongly-algebrable.
\end{theorem}

\begin{proof}[Proof]
For any $n \in \N$, $p > 0$, we consider the ``triangles'' on $[0,+\infty)$ given by
$$T_{n,p}(x) =
\begin{cases}
n^p(2^{n+1}x + (1-n2^{n+1}))^p  & \text{if } x \in \left[n-\tfrac{1}{2^{n+1}},n\right), \\
n^p(-2^{n+1}x + (1+n2^{n+1}))^p & \text{if } x \in \left[n,n+\tfrac{1}{2^{n+1}}\right], \\
0                           & \text{otherwise},
\end{cases}
$$
and we define the functions $g_p : [0,+\infty) \longrightarrow \R$ as:
$$g_p(x) = \sum_{n=1}^{\infty} T_{n,p}(x).$$

Because of the disjointness of the supports of the triangles $T_{n,p}(x)$ ($n \in \N$), the functions  $g_p$ are well defined and continuous on $[0,+\infty)$. Furthermore, we can easily bound their $L^1$-norm. Observe that the triangles $T_{n,p}(x)$ are always ``inside'' the real triangles of basis $\left[n-\frac{1}{2^{n+1}},n+\frac{1}{2^{n+1}}\right]$ and height $n^p$ (on $x=n$). Hence
$$\|g_p\|_{L^1} =  \sum_{n=1}^{\infty}\int_0^{\infty}  T_{n,p}(x)dx \leq \sum_{n=1}^{\infty} \dfrac{n^p}{2^n} < +\infty.$$
So,  $g_p \in \mathcal{C}([0,+\infty)) \cap L^1([0,+\infty))$. Now, by evaluating $g_p$ on the sequence $(x_n)_n = (n)_n$ we get that
$$g_p(x_n) = g_p(n) = n^p \overset{n\to\infty}{\longrightarrow} +\infty,$$
which yields $g_p \in \mathcal{A}$ for all $p > 0$.

Let $(p_j)_j$ be the increasing sequence of prime numbers, and let us define for each $j \in \N$ the function $F_j$ by
$$F_j(x) := \sum_{n=1}^{\infty} T_{n,\log p_j}(x) = g_{\log p_j}(x).$$
Note that $(F_j)_j\subset\mathcal{A}$. Let $\mathcal{B}$ be the algebra generated by $(F_j)_{j}$, that is,
$$\begin{array}{rl}
\mathcal{B} = \big\{ P(F_{1}, \ldots, F_{s}):& P\text{ is a polynomial in } s \text{ variables without}\\
&\text{constant term},\,  s \in \N \big\}.\end{array}$$

We are going to prove that $\mathcal{B}$ is the desired infinitely generated algebra in $\mathcal{A}$.
Let $m(x_1, \ldots, x_s) = \prod_{i=1}^s x_i^{\alpha_i}$ ($\alpha_i \in \N\cup\{0\}$) be a non constant monomial. Hence
$$m(F_{1}, \ldots, F_{s})(x) = \prod_{i=1}^s F_{i}(x)^{\alpha_i} = \prod_{i=1}^s \left( \sum_{n=1}^{\infty} T_{n,\log(p_{i})}(x) \right)^{\alpha_i}.$$

For each $n \in \N$ and each $x \in \left[n-\frac{1}{2^{n+1}},n\right)$ we have
\begin{eqnarray*}
m(F_{1}, \ldots, F_{s})(x) &=& \prod_{i=1}^s T_{n,\log(p_{i})}(x)^{\alpha_i} \\
                                        &=&\prod_{i=1}^s\left(\left[n\left(2^{n+1}x+(1-n2^{n+1})\right)\right]^{\log(p_{i})}\right)^{\alpha_i} \\
                                        &=& \left[ n\left(2^{n+1}x+(1-n2^{n+1})\right)\right]^{{ \sum_{i=1}^s \log(p_{i}^{\alpha_i})}} \\
                                        &=& {{\left[ n\left(2^{n+1}x+(1-n2^{n+1})\right)\right]}}^{{ \log\left( \text{ind}_p(m)\right)}}
\end{eqnarray*}
Analogously, for each $x \in \left[n,n+\frac{1}{2^{n+1}}\right]$, we get
$$m(F_{1},\ldots, F_{s})(x) = \left[ n\left(-2^{n+1}x+(1+n2^{n+1})\right)\right]^{{ \log\left( \text{ind}_p(m)\right)}}.$$
So
\begin{equation}\label{m es g}
m(F_{1},\ldots, F_{s}) =g_{\log({\rm ind}_p(m))}.\end{equation}
In particular, $m(F_{1}, \ldots, F_{s}) \in \mathcal{A}$, and trivially $\mathcal{B} \subset \mathcal{C}([0,+\infty)) \cap L^1([0,+\infty))$. It only remains to prove that $(F_j)_j$ forms an algebraic independent set and that any element of $\mathcal{B}$ is not bounded. In order to prove these properties, let us take a nontrivial algebraic combination
$$F(x) := P(F_{1}, \ldots, F_{s})(x) = \sum_{i=1}^l \lambda_i m_i(F_{1}, \ldots, F_{s})(x).$$
For each $n \in \N$, we have by \eqref{m es g} that
\begin{equation}\label{polinomio}
F(n) = P(F_{1}, \ldots, F_{s})(n) = \sum_{i=1}^l \lambda_i n^{\log(\text{ind}_p(m_i))}.
\end{equation}
By Remark \ref{obs1}, all exponents of the right hand part of \eqref{polinomio} are positive and pairwise different, so $|F(n)| \overset{n\to\infty}{\longrightarrow}+\infty$. Hence, $F \in \mathcal{A}$, and $\mathcal{B}\setminus\{0\} \subset \mathcal{A}$.

Finally, it is also clear from \eqref{polinomio} and Remark \ref{obs1} that $(F_j)_j$ is a free generated system of the algebra $\mathcal{B}$, and thus $\mathcal{A}$ is strongly algebrable.

\end{proof}

\begin{remark}\label{residual}
  In a purely topological sense, it is easy to see that $\mathcal{A}$ is large. To be more specific, $\mathcal{A}$ is residual in our Baire space $X$ (note that $X$ is Baire because, as it is easy to see, the distance $d_X$ given in \eqref{metrica} is complete). Indeed, $$X\setminus\mathcal{A}=\bigcup_{n=1}^\infty \mathcal{A}_n,$$ where $$\mathcal{A}_n:=\left\{f\in X:\, |f(x)|\le n\text{ for all }x\ge 0\right\}.$$
  Since uniform convergence on compacta implies pointwise convergence, it is clear that each $\mathcal{A}_n$ is closed in $X$. Moreover, given $f\in\mathcal{A}_n$ and $\eps>0$, it is also easy but cumbersome to construct a function $g\in\mathcal{A}$ such that $d_X(f,g)<\eps$ but $|g(x_0)|>n$ for some $x_0>0$. This implies that every $\mathcal{A}_n$ has empty interior, so that $X\setminus\mathcal{A}$ is of firste category, which proves that $\mathcal{A}$ is residual.
\end{remark}

%%%%%%%%%%%%%%%%%%%%%%%%%%%%%%%%%%%%%%%%%%%%%%%%%%%%%%%%%%%%%%%%%%%%%%%%%%%%%%%

\section{Sequences of unbounded, continuous and integrable functions}

We denote by $\mathcal{A}_0$ the space of sequences of unbounded, continuous and integrable functions in $[0,+\infty)$, converging to zero in the metric $d_X$ defined by \eqref{metrica}, that is,
$$\mathcal{A}_0 := \{(f_n)_n: \, f_n \in \mathcal{A} \text{ for all }n\in\N \text{ and } d_X(f_n,0) \overset{n\to\infty}{\longrightarrow}0\}.$$

In order to prove the dense-lineability of $\mathcal{A}_0$, we consider the sequence space
$$c_0(X)    := \{ (f_n)_n :\, f_n \in X \text{ for all } n \in \N \text{ and } d_X(f_n,0)  \overset{n\to\infty}{\longrightarrow}0\},$$
endowed with the following distance:
\begin{equation}\label{metrica sucesiones} d_{c_0(X)}((f_n)_n, (g_n)_n) = \sup_{n \in \N} d_X(f_n, g_n).\end{equation}
Observe that $(c_0(X), d_{c_0(X)})$ is a metric topological vector space and that, given a sequence $((f_{k,n})_n)_k\subset c_0(X)$, then $d_{c_0(X)}(((f_{k,n})_n)_k,0) \rightarrow 0$ ($k\to\infty$) if and only if $\sup_n|f_{k,n}|\to0$ ($k\to\infty$) uniformly on compacta in $[0,+\infty)$.

We define
\begin{eqnarray*}
c_{00}(B) :=
 \big\{ (b_n)_n&:& \text{exists } n_0\in\N \text{ such that } b_n \in B \text{ for all }n \leq n_0 \text{ and }\\
 && b_n = 0 \text{ for all } n > n_0\big\},
 \end{eqnarray*}
where $B$ is the dense subset of $X$ defined in the proof of Theorem \ref{T3}.

\begin{lemma}\label{denalg}
The space $c_{00}(B)$ is dense in $c_0(X)$.
\end{lemma}

\begin{proof}
Let $(f_n)_n\in c_0(X)$. Given $\eps>0$, there is $n_0\in\N$ with $d_X(f_n,0)<\eps$ for all $n\ge n_0$. Furthermore, the denseness of $B$ in $X$ guarantees the existence of $b_1,\ldots,b_{n_0-1}\in B$ such that $d_X(f_n,b_n)<\eps$ ($n=1,\ldots,n_0-1$). Finally, define $b_n:=0$ for $n\ge n_0$. It is clear that $(b_n)_n\in c_{00}(B)$ and, by construction, $$d_{c_0(X)}((f_n)_n,(b_n)_n) = \max\left\{\sup_{1\le n\le n_0-1} d_X(f_n,b_n),\sup_{n\ge n_0} d_X(f_n,0)\right\}<\eps,$$ and we are done.
\end{proof}

Now we can establish our first result on sequences spaces.

\begin{theorem}\label{S0}
The family of sequences $\mathcal{A}_0$ is maximal dense-lineable.
\end{theorem}

\begin{proof}[Proof]
Consider the sequence of functions $(f_n)_n$ given by
$$f_n(x) := \sum_{m=n}^{\infty} T_m(x), \quad (x \geq 0,\, n\in\N)$$
where $T_m$ is the isosceles triangle of height $m$ whose basis is the interval $\left[m-\frac{1}{2^{m+1}},m+\frac{1}{2^{m+1}}\right]$.

Observe that, again, we have a gap of at least $\frac{1}{8}$ between each pair of consecutive triangles of $f_n$. Now define the functions
$$f_{n,t}(x) := f_n(x-t), \quad t \in \left[0,\tfrac{1}{8}\right),$$
and consider the set $M_0$ given by
$$M_0 := \text{span} \left\{ (f_{n,t})_n \ :  t \in \left[0,\tfrac{1}{8}\right)\right\}.$$

Following the same argument as in Theorem \ref{T1}, we have that each $f_{n,t}$ is a continuous, unbounded and integrable function in $[0,+\infty)$, that is, $f_{n,t} \in \mathcal{A}$ for all $t \in \left[0,\frac{1}{8}\right)$ and $n \in \N$, and that the sequences $(f_{n,t})_n$ are linearly independent.

In addition, it is clear that the whole series $\dis\sum_{m=1}^\infty T_m(x)$ converges to zero both in $L_1$-norm an uniformly in compact sets of $[0,+\infty)$, so in $d_X$. Hence $d_X(f_{n,t},0)\to0$ ($n\to\infty$) for all $t\in[0,\frac{1}{8})$,  $M_0\setminus\{0\}\subset{\mathcal A}_0$ and ${\mathcal A}_0$ is maximal lineable.

Now, by Lemma \ref{denalg}, $c_{00}(B)$ is a dense-lineable subset of $c_0(X)$. Moreover, each function $b_n$ of a sequence $(b_n)_n \in c_{00}(X)$ is bounded, so $c_{00}(B) + \mathcal{A}_0 \subset \mathcal{A}_0$. Finally, Lemma \ref{denso} tells us that $\mathcal{A}_0$ is maximal dense-lineable.
\end{proof}

As in the case of single functions, we now can prove the algebrability of ${\mathcal A}_0$.

\begin{theorem}
The family of sequences $\mathcal{A}_0$ is strongly algebrable.
\end{theorem}

\begin{proof}[Proof] Let $F_j$ ($j\in\N$) be the functions constructed in the proof of Theorem \ref{T3}, that is, $F_j(x) = \sum_{n=1}^{\infty} T_{n,\log(p_j)}(x) = g_{\log(p_j)}$, where $(p_j)_j$ is the increasing sequence of prime numbers.

For each $j,n\in\N$, let $F_{j,n}(x) := \sum_{m=n}^{\infty} T_{m,\log(p_j)}(x)$. Let $\mathcal{B}_0$ be the algebra generated by the sequences $\{(F_{j,n})_n:\, j\in\N\}$.

Following the same argument as in Theorem \ref{T3} and taking the monomial $M(x_1, \ldots, x_s) = \prod_{i=1}^s x_i^{\alpha_i}$, we have for each $n \in \N$ that
\begin{equation}\label{monomio2}
M(F_{1,n}, \ldots, F_{s,n})(x) = \sum_{m=n}^\infty T_{m,\log({\rm ind}_p(M))}(x).\end{equation}
Therefore, each component of the sequence $M((F_{1,n})_n,\ldots,(F_{s,n})_n)$ is the tail of a convergent series (in the topology generated by $d_X$) and we have that $d_X(M(F_{1,n},\ldots,F_{s,n}),0)\overset{n\to\infty}{\longrightarrow} 0$. In particular, $M((F_{1,n})_n, \ldots, (F_{s,n})_n)\in {\mathcal A}_0$.

Finally, let $(F_n)_n$ be a nontrivial algebraic combination of the sequences $(F_{j,n})_n$ $(j\in\N)$, that is, for each nonzero polynomial $P$ in $s$ variable ($s\in\N$) and each $n\in\N$, we consider $$F_n(x):=P(F_{1,n},\ldots,F_{s,n})(x)=\sum_{i=1}^l \lambda_i M_i(F_{1,n},\ldots,F_{s,n})(x),$$ where the $\lambda_i$'s are not simultaneously zero.
For each $n\in\N$ and each $m\ge n$, taking \eqref{monomio2} into account, we obtain
\begin{equation*}
  \label{monomiofinal}
  F_n(m)=P(F_{1,n},\ldots,F_{s,n})(m)=\sum_{i=1}^l\lambda_i m^{\log({\rm ind}_p(M_i))}
\end{equation*}
Now, continuing in the same manner as the proof of Theorem \ref{T3}, we get that
$|F_n(m)|\to+\infty$ ($m\to\infty$), hence ${\mathcal B}_0\setminus\{0\}\subset{\mathcal A}_0$ and $\mathcal{B}_0$ is a freely generated algebra. Thus $\mathcal{A}_0$ is strongly-algebrable.
\end{proof}

\section{Further results and final remarks}
1. Observe that in the proof of Theorem \ref{S0} we have sequences converging to zero not only uniformly on compacta, but also almost uniformly to zero. Recall that $f_n \to f$ {\em almost uniformly} (on $[0,+\infty)$) if, for every $\varepsilon > 0$, there exists a set $E\subset[0,+\infty)$ with $m(E)\leq \varepsilon$ such that $f_n\to f$ uniformly on $[0,+\infty)\setminus E$, where $m$ denotes the Lebesgue measure. Indeed, following the notation of the proof of Theorem \ref{S0}, we can define for each $n\in\N$ the set
$$
E_n := \text{support}(f_{n,t}) = \bigcup_{m = n}^{\infty} \left( 1+m+t- \dfrac{1}{2^{m+1}}, 1 +m+t+ \dfrac{1}{2^{m+1}} \right).$$
It is clear that, for each $n\in\N$, $E_{n+1}\subset E_n$ and $m(E_n) = \sum_{m=n}^{\infty} \frac{1}{2^m} \to0$ ($n\to\infty$).

So, given $\varepsilon > 0$, it is possible to find $N \in \N$ such that $m(E_N) < \varepsilon$ and $f_{n,t}(x)=0$ for all $n\ge N$ and all $x\in[0,+\infty)\setminus E_N$.

 In this sense, the result in Theorem \ref{S0} is sharp, since it is not possible to get the uniform convergence, due to the unboundedness condition. Moreover, we cannot get the uniform convergence almost everywhere in $[0,+\infty)$ (where we say that {\em $f_n\to f$ uniformly almost everywhere} whenever there is $E\subset[0,+\infty)$ with $m(E)=0$ such that $f_n\to f$ uniformly on $[0,+\infty)\setminus E$) as is proved in the next proposition.
    \begin{proposition}
    Let $(f_n)_n \subset \mathcal{C}([0,+\infty))$ such that $f_n \to 0$ pointwise on $[0,+\infty)$ and each $f_n$ is unbounded. Then $f_n \not \to 0$ uniformly almost everywhere in $[0,+\infty)$.
    \end{proposition}

    \begin{proof} By way of contradiction, assume that there exists a set $E \subset [0,+\infty)$ with $m(E) = 0$ such that $f_n\to 0$ uniformly in $[0,+\infty) \setminus E$. Then, for $\varepsilon = 1$, there exists $n_0 \in \N$ such that for all $n \geq n_0$ and all $x \in [0,+\infty) \setminus E$ it holds that $|f_n(x)| \leq 1$.

    Since $f_n$ is unbounded for each $n \in \N$, there exists $x_n\in[0,+\infty)$ such that $|f_n(x_n)|>2$, so $x_n\in E$, for $n\ge n_0$. Now, the continuity of each $f_n$ and the fact that $m(E)=0$ guarantee the existence of points $w_n\notin E$ but near enough to $x_n$ (note that $m(E)=0$ implies the denseness of $[0,+\infty)\setminus E$) such that $|f_n(w_n) - f_n(x_n)| < \frac{1}{2}$. But then, for $n\ge n_0$ we get
    $$1\ge |f_n(w_n)|\ge|f_n(x_n)|-|f_n(x_n)-f_n(w_n)|>2-\frac{1}{2}=\frac{3}{2},$$
    which is a contradiction.
    \end{proof}

\bigskip

2. Let $\alpha : [0,+\infty) \longrightarrow [1,+\infty)$ be a continuous and nondecreasing function. We say that $f \in \mathcal{C}([0,+\infty))$ has {\em growth} $\alpha$ if
$$\limsup_{x \rightarrow +\infty} \dfrac{|f(x)|}{\alpha(x)} = +\infty.$$

If we modify the sequence constructed in Example \ref{ejemplo1} (and Theorems \ref{T1} and \ref{T2}) by using triangles of height $\alpha(n)$ and bases $\frac{1}{\alpha(n)2^n}$ and, in the same way, the curved triangles of Theorem \ref{T3}, we can adapt the proof of all results of Section 2 to get the next more general theorem.

\begin{theorem}\label{T4} Let $\alpha : [0,+\infty) \longrightarrow [1,+\infty)$ be a continuous and nondecreasing function. The family $\mathcal{A}(\alpha)$ of continuous and integrable functions in $[0,+\infty)$ with growth $\alpha$ is maximal dense-lineable and strongly algebrable.
\end{theorem}

The last theorem is a generalization of those in Section 2 just by taking $\alpha(x) \equiv 1$. Furthermore, we are able to state a large algebraic structure of families of continuous and integrable functions that grow exponentially or even faster: choose $\alpha(x) = e^x$, $\alpha(x) = e^{e^x}$, etc.
%\end{remarks}

Following the same steps as in Section 3, and taking into account the preceding remark, we can also establish the large algebraic structure of the set of sequences of functions in ${\mathcal A}(\alpha)$ converging to zero in $L^1$-norm and almost uniformly in $[0,+\infty)$.

\begin{theorem}\label{T5} Let $\alpha : [0,+\infty) \longrightarrow [1,+\infty)$ be a continuous and nondecreasing function. The family ${\mathcal A}_0(\alpha)$ of sequences $(f_n)_n$ of continuous and integrable functions with growth $\alpha$ such that $f_n\to0$ in $L_1$-norm and almost uniformly in $[0,+\infty)$ is maximal dense-lineable and strongly algebrable.
\end{theorem}

\bigskip

3. If we modify the construction of Example \ref{ejemplo1} by changing the triangles $T_n$ to exponential-like functions, for instance $$\Phi_n(x):=\begin{cases}
  n\cdot e^{%\left(
  1-\frac{1}{1-\left(2^{n+1}(x-n)\right)^2}%\right)
  }&\text{if }x\in \left( n-\frac{1}{2^{n+1}}, n+\frac{1}{2^{n+1}}\right)\\0&\text{otherwise,}
\end{cases}$$ we get for every $n\in\N$ an integrable, unbounded and ${\mathcal C}^\infty$ function on $[0,+\infty)$. Hence a similar argument to the one given in Theorem \ref{T1} allows us to state the maximal lineability of
$$
{\mathcal A}^\infty :=\big\{f\in{\mathcal C}^\infty([0,+\infty))\cap L^1([0,+\infty)):\, \limsup_{x\to+\infty}|f(x)|=+\infty\big\}.$$
Moreover, because $\|\Phi_n\|_{L^1}\le \frac{1}{2^n}$, we also obtain the maximal lineability of the set ${\mathcal A}_0^\infty$ of sequences $(f_n)_n$ of functions in ${\mathcal A}^\infty$ such that $f_n\to0$ in $L^1$-norm and almost uniformly in $[0,\infty)$. For the dense-lineability of both sets, we need first to state an adequate metric in $Y:=\mathcal{C}^\infty([0,+\infty))\cap L^1([0,+\infty))$ and $c_0(Y)$; so we left it, together with the strong algebrability, as an open problem.

\section*{Acknowledgment}
The authors would kindly acknowledge Prof.~Bernal-Gonz\'alez for his helpful and useful comments and suggestions.


\begin{thebibliography}{99}
%\begin{bibdiv}
%\begin{biblist}

\bibitem{albuquerque2014} {N. Albuquerque},
{Maximal lineability of the set of continuous surjections},
{Bull. Belg. Math. Soc. Simon Stevin}
{\bf 21}
{(2014)},
{83--87}.


\bibitem{albuquerquebernalpellegrinoseoane2014}
{N. Albuquerque},
{L. Bernal-Gonz\'alez},
{D. Pellegrino}
{and J.B. Seoane-Sep\'ulveda},
{Peano curves on topological vector spaces},
{Linear Algebra Appl.}
{\bf 460}
{(2014)},
{81--96}.


\bibitem{araujobernalmunozpradoseoane2017}
{G. Ara\'ujo},
{L. Bernal-Gonz\'alez},
{G.A. Mu\~noz-Fern\'andez},
{J.A. Prado-Bassas}
{and J.B. Seoane-Sep\'ulveda},
{Lineability in sequence and function spaces},
{Studia Math.}
{\bf 237}
{(2017)},
{119--136}.

\bibitem{aronbernalpellegrinoseoane2015}
{R.M. Aron},
{L. Bernal-Gonz\'alez},
{D. Pellegrino}
{and J.B. Seoane-Sep\'ulveda},
{Lineability: The search for linearity in Mathematics, Monographs and Research Notes in Mathematics},
{Monographs and Research Notes in Mathematics},
{Chapman \& Hall/CRC},
{Boca Raton, FL},
{2016}.

\bibitem{arongarciaperezseoane2009}
{R.M. Aron},
{F.J. Garc\'\i a-Pacheco},
{D. P\'erez-Garc\'\i a}
{and J.B. Seoane-Sep\'ulveda},
{On dense-lineability of sets of functions on {$\mathbb R$}},
{Topology}
{\bf 48}
{(2009)},
{149--156}.

\bibitem{aronperezseoane2006}
{R.M. Aron},
{D. P\'erez-Garc\'{\i}a}
{and J.B. Seoane-Sep\'ulveda},
{Algebrability of the set of nonconvergent Fourier series},
{Studia Math.}
{\bf 175}
{(2006)},
{83--90}.


\bibitem{bayartquarta2007}
{F. Bayart}
{and L. Quarta},
{Algebras in sets of queer functions},
{Israel J. Math.}
{\bf 158}
{(2007)},
{285--296}.

\bibitem{bernal2010}
{L. Bernal-Gonz\'alez},
{Algebraic genericity of strict-order integrability},
{Studia Math.}
{\bf 199}
{(2010)},
{279--293}.


\bibitem{bernalcalderonprado2015}
{L. Bernal-Gonz\'alez},
{M.C. Calder\'on-Moreno}
{and J.A. Prado-Bassas},
{The set of space-filling curves: topological and algebraic structure},
{Linear Algebra Appl.}
{\bf 467}
{(2015)},
{57--74}.


\bibitem{bernalordonez2013}
{L. Bernal-Gonz\'alez}
{and M. Ord\'o\~nez-Cabrera},
{Spaceability of strict order integrability},
{J. Math. Anal. Appl.}
{\bf 385}
{(2012)},
{303--309}.
		

\bibitem{bernalordonez2014}
{L. Bernal-Gonz\'alez}
{and M. Ord\'o\~nez-Cabrera},
{Lineability criteria, with applications},
{J. Funct. Anal.}
{\bf 266}
{(2014)},
{3997--4025}.
		

\bibitem{bernalpellegrinoseoane2014}
{L. Bernal-Gonz{\'a}lez},
{D. Pellegrino}
{and J.B. Seoane-Sep{\'u}lveda},
{Linear subsets of nonlinear sets in topological vector spaces},
{Bull. Amer. Math. Soc. (N.S.)}
{\bf 51}
{(2014)},
{71--130}.

\bibitem{botelhocariellofavaropellegrinoseoane2013}
{G. Botelho},
{D. Cariello},
{V.V. F\'avaro},
{D. Pellegrino}
{and J.B. Seoane-Sep{\'u}lveda},
{Distinguished subspaces of {$L_p$} of maximal dimension},
{Studia Math.}
{\bf 215}
{(2013)},
{261--280}.

\bibitem{botelhofavaropellegrinoseoane2012}
{G. Botelho},
{V.V. F\'avaro},
{D. Pellegrino}
{and J.B. Seoane-Sep{\'u}lveda},
{{$L_p[0,1]\setminus\bigcup_{q>p}L_q[0,1]$} is spaceable for every {$p>0$}},
{Linear Algebra Appl.}
{\bf 436}
{(2012)},
{2963--2965}.

\bibitem{garciagrecumaestreseoane2010}
{D. Garc\'\i a},
{B.C. Grecu},
{M. Maestre}
{and J.B. Seoane-Sep\'ulveda},
{Infinite dimensional {B}anach spaces of functions with nonlinear properties},
{Math. Nachr.}
{\bf 283}
{(2010)},
{712--720}.

\bibitem{garciamartinseoane2009}
{F.J. Garc\'\i a-Pacheco},
{M. Mart\'\i n}
{and J.B. Seoane-Sep\'ulveda},
{Lineability, spaceability, and algebrability of certain subsets of function spaces},
{Taiwanese J. Math.}
{\bf 13}
{(2009)},
{1257--1269}.

\bibitem{gurariy1996}
{V.I. Gurariy},
{Subspaces and bases in spaces of continuous functions}
{(Russian)},
{Dokl. Akad. Nauk SSSR}
{\bf 167}
{(1966)},
{971--973}.

\bibitem{jimenezmunozseoane2013}
{P. Jim\'enez-Rodr\'\i guez},
{G.A. Mu\~noz-Fern\'andez}
{and J.B. Seoane-Sep\'ulveda},
{On {W}eierstrass' {M}onsters and lineability},
{Bull. Belg. Math. Soc. Simon Stevin}
{\bf 20}
{(2013)},
{577--586}.

\bibitem{kitsontimoney2011}
{D. Kitson}
{and R.M. Timoney},
{Operator ranges and spaceability},
{J. Math. Anal. Appl.}
{\bf 378}
{(2011)}
{680--686}.

\bibitem{munozpalmbergpuglisiseoane2008}
{G.A. Mu\~noz-Fern\'andez},
{N. Palmberg},
{D. Puglisi}
{and J.B. Seoane-Sep\'ulveda},
{Lineability in subsets of measure and function spaces},
{Linear Algebra Appl.}
{\bf 428}
{(2008)},
{2805--2812}.

\end{thebibliography}
\end{document}